\documentclass[%review,
                final,             
               onefignum,
               onetabnum
                        ]{siamart250211}
\usepackage{amsmath}
\usepackage{upref} % Upright \ref.
\usepackage{amssymb}
\usepackage{breakcites}
\usepackage{xcolor}
\usepackage{array}
\usepackage{calc}
\usepackage{booktabs}
\usepackage{siunitx}
\usepackage{float} 
\floatstyle{plain}
\usepackage{algorithm,algpseudocode}
\usepackage{caption}
\usepackage{graphicx}
\usepackage{url}
\usepackage{placeins} %To avoid weird grouping of figures
\usepackage{fancyhdr,substr}  % For git header.
\usepackage{gitinfo}

\usepackage[english]{babel}

\usepackage[markup=defult]{changes}
\definechangesauthor[name={Awad}, color=red]{AA}

\ifpdf
  \DeclareGraphicsExtensions{.eps,.pdf,.png,.jpg}
\else
  \DeclareGraphicsExtensions{.eps}
\fi
\makeatletter
\def\@cite#1#2{[{#1\if@tempswa ,~#2\fi}]}% NEW
\makeatother
\headers{Exponential of a Block Triangular Matrix}{Awad H. Al-Mohy}
\title{A New Algorithm for Computing the Exponential of a Block Triangular Matrix
       \thanks{Version of \today
       \funding{This work was supported by the Deanship of Scientific Research at King Khalid University Research Groups Program (grant RGP. 1/318/45).
               }
              }
      }

\author{Awad H. Al-Mohy%
           \thanks{%
                   Department of Mathematics, King Khalid University, Abha,
                   Saudi Arabia
                   (\email{ahalmohy@kku.edu.sa}%,
                    ).
        }
}

\makeatletter
\def\argmin{\mathop{\operator@font argmin}}
\makeatother

\def\funm{\t{funm}}
\def\expm{\t{expm}}
\def\expmmp{\t{expm\_mp}}
\def\sylvester{\t{sylvester}}
\def\testmats{\texttt{testmats}}
\def\randn{\texttt{randn}}

\def\At{\widetilde{A}}
\def\Tt{\widetilde{T}}
\def\Ht{\widetilde{g}}

\def\dt{\widetilde{d}}

\def\refsol{\mathop{\mathrm{ref}}}

\def\C{\mathbb{C}}

\def\vec{\mathop{\mathrm{vec}}}

\def\and{\mathop{\mathrm{and}}}

\def\Htmd{\Ht@@'_{2m+1}}
\def\Pant{Pad\'e approximant}
\def\FD{Fr\'echet derivative}

\mathcode`@="8000 % Make @ behave as per catcode 13 (active).  TeXbook p. 155.
{\catcode`\@=\active\gdef@{\mkern1mu}}
\def\be{backward error}

\def\alg{algorithm}

\def\e{eigenvalue}
\def\t#1{\texttt{\upshape #1}}
\def\eps{\epsilon}

\newcommand{\myfig}[2][12cm]{\includegraphics[width = #1]{fig/#2.eps}}
\newcounter{mylineno}
\makeatletter
\let\oldtabcr\@tabcr

\def\mynewline{\refstepcounter{mylineno}%
                \llap{\footnotesize\arabic{mylineno}\hspace{5pt}}%
               }

\gdef\@tabcr{\@stopline \@ifstar{\penalty%
            \@M \@xtabcr}\@xtabcr\mynewline}

\makeatother
\def\a{\alpha}

\def\D{\Delta}

\def\resp{respectively}

\def\Dhat{\widehat{D}}

\def\cn{condition number}

\def\tol{\mathrm{tol}}
\def\cond{\mathrm{cond}}

\def\vec{\mathrm{vec}}

\def\antidiag{\mathrm{antidiag}}
\def\dt{\mathrm{d}}
\def\1i{\mathrm{i}}
\def\e{\mathrm{e}}

\def\norm#1{\|#1\|}
\def\normt#1{\|#1\|_2}
\def\normF#1{\|#1\|_F}

\def\normi#1{\|#1\|_1}

\def\nbyn{n \times n}

\def\C{\mathbb{C}}
\def\R{\mathbb{R}}

\def\cL{\mathcal{D}}
\let\oldref\ref
\def\ref#1{{\normalfont\oldref{#1}}}
\def\eqref#1{{\normalfont(\oldref{#1})}}
\newcounter{exprmt}
\def\Experiment{Experiment}

               {\begin{list}{\indent\emph{\Experiment} {\upshape\arabic{exprmt}.}}%
                {\usecounter{exprmt}
                \setlength{\leftmargin}{\rightmargin}
                \setlength{\labelwidth}{\leftmargin}
                \addtolength{\labelwidth}{-\labelsep}
                \setlength{\topsep}{0in}
                \setlength{\itemsep}{0pt}
                \setlength{\listparindent}{\parindent}%
               }}%
               {\end{list}}
\mathchardef\Gamma="7100 \mathchardef\Delta="7101
\mathchardef\Theta="7102 \mathchardef\Lambda="7103
\mathchardef\Xi="7104 \mathchardef\Pi="7105 \mathchardef\Sigma="7106
\mathchardef\Upsilon="7107 \mathchardef\Phi="7108
\mathchardef\Psi="7109 \mathchardef\Omega="710A
\newsiamthm{myalg}{\textbf{Algorithm}}
\begin{document}
\maketitle
\vspace{0.5cm}

\begin{center}
    \textit{Dedicated to the memory of Nick Higham}
\end{center}

\vspace{0.5cm}
\begin{abstract}
The exponential of block triangular matrices arises in a wide range of scientific computing applications, including exponential integrators for solving systems of ordinary differential equations, Hamiltonian systems in control theory, sensitivity analysis, and option pricing in finance.
We propose a novel algorithm exploiting the block triangular structure for simultaneously computing the exponentials of the diagonal blocks and the off-diagonal block of the matrix exponential without direct involvement of the full block matrix in the computations. This approach generalizes the work of Al-Mohy and Higham [\emph{SIAM J. Matrix Anal. Appl.}, 30 (2009), pp. 1639–1657] on the Fr\'echet derivative of the matrix exponential. The generalization is established through a linear operator framework, facilitating efficient evaluation schemes and rigorous backward error analysis.
The algorithm employs the scaling and squaring method using diagonal Pad\'e approximants with algorithmic parameters selected based on the backward error analysis. A key feature is that the selection of the scaling parameter relies solely on the maximal norm of the diagonal blocks with no dependence on the norm of the off-diagonal block.
Numerical experiments confirm that the proposed algorithm consistently outperforms existing algorithms in both accuracy and efficiency, making it a preferred choice for computing the matrix exponential of block triangular matrices.

\end{abstract}

%%%%%%%%%%%%%%%%
%%%%%%%%%%%%%%%%
\section{Introduction}
Matrix functions play a fundamental role in many scientific and engineering disciplines, as they provide a structured approach to solving complex problems. They have broad applications in fields such as control theory, quantum mechanics, network analysis, and machine learning \cite{ehh08,high:FM,hial10}. As a result, matrix functions have attracted substantial research attention, leading to significant advances in theoretical understanding and computational methods.
In particular, the matrix exponential is one of the most extensively studied matrix functions, both from theoretical and computational perspectives. However, at the algorithmic level, the computation of the matrix exponential for some structured problems has received comparatively less attention, despite its importance in applications where exploiting matrix structure can lead to substantial gains in efficiency and accuracy.
In this context, and motivated by applications, we propose a specialized algorithm for computing the exponential of block triangular matrices that fully exploits the triangular structure without explicitly forming the full matrix.

The contents of this paper are organized as follows. First, for a given analytic function $f$ and matrices $A$ and $B$ of appropriate dimensions, we define a linear operator via the Cauchy integral and derive its fundamental properties. Motivated by the potential of this operator to facilitate the computation of the exponential of block triangular matrices, we highlight several applications. In section~\ref{sect.eval.exp}, we apply the operator to the optimally formulated evaluation schemes for the diagonal Pad\'e approximants to the exponential function, thereby obtaining optimal evaluation schemes for the operator itself, which, in this setting, represents the off-diagonal block of the matrix exponential. This approach leads to significant computational savings and improved efficiency. The scaling and squaring method is then employed and analyzed.
We present a rigorous backward error analysis in section~\ref{sect.be.anal}. In section~\ref{sect.algs}, we describe our proposed algorithm and review the algorithm of Kenney and Laub. Numerical experiments are provided in section~\ref{sect.num}. Finally, we draw some concluding remarks in section~\ref{sect.conc}. and raise some open questions.

Unless otherwise stated, the matrix norms used in this paper are assumed to be consistent.
\begin{definition}
\label{def1}
Let $f$ be an analytic function on a simply connected open set $\Omega\subset\C$ containing the union of the spectra of $A\in\C^{n\times n}$ and $B\in\C^{d\times d}$. For every $E\in\C^{n\times d}$ we define the operator $\cL_f(A,B,\cdot):\C^{n\times d}\to\in\C^{n\times d}$ as
\begin{equation}\label{Lab.operator}
 \cL_f(A,B,E) = \frac{1}{2\pi\1i}\int_\Gamma f(z)(zI_n-A)^{-1}E(zI_d-B)^{-1}\dt z,
\end{equation}
where $\Gamma\subset\Omega$ is a simple closed rectifiable contour that strictly encloses the union of the spectra of $A$ and $B$ and winds once around it in the counterclockwise direction.
% and traverses once in the counterclockwise direction.
\end{definition}
Obviously $\cL_f(A,B,\cdot)$ is linear and uniquely determined by the function $f$ and the ordered pair $(A,B)$.
This definition follows immediately from the Cauchy integral definition \cite[Def.~1.11]{high:FM} of the matrix function
$f\left(\begin{bmatrix}
         A  &  E \\ 0 & B
\end{bmatrix}\right)$. That is,
\begin{eqnarray}
\nonumber
 f\left(\begin{bmatrix}
         A  &  E \\ 0 & B
\end{bmatrix}\right) &=& \frac{1}{2\pi\1i}\int_\Gamma f(z)
\begin{bmatrix}
         zI_n-A  &  -E \\ 0 & zI_d-B
\end{bmatrix}^{-1} \dt z\\
\nonumber
                     &=& \frac{1}{2\pi\1i}\int_\Gamma f(z)
\begin{bmatrix}
         (zI_n-A)^{-1}  &  (zI_n-A)^{-1}E(zI_d-B)^{-1}  \\ 0 & (zI_d-B)^{-1}
\end{bmatrix} \dt z\\ &=&
\begin{bmatrix}
         f(A)  &  \cL_f(A,B,E) \\ 0 & f(B)
\end{bmatrix}.
\label{block.def}
\end{eqnarray}
 The next lemma lists several properties and rules of this operator.
\begin{lemma}
\label{lem.prop}
Let \( f \) and \( g \) be matrix functions satisfying the assumptions of Definition \ref{def1}. Then we have:

\begin{enumerate}
    \item \( \cL_f(A, A, E) = L_f(A, E) \), where \( L_f(A, \cdot) \) is the \FD\ operator of the matrix function \( f \) \textnormal{\cite[Eq.~(3.6)]{high:FM}}.
    
    \item \( \cL_f(A, B, E) = P_1 \cL_f(T_1, T_2, P_1^{-1} E P_2) P_2^{-1} \), where \( A = P_1 T_1 P_1^{-1} \) and \( B = P_2 T_2 P_2^{-1} \).
    
    \item \( \cL_f(A, B, E) = 0 \) if \( f \) is constant.
    
    \item \( \cL_{f+g}(A, B, E) = \cL_f(A, B, E) + \cL_g(A, B, E) \) \hfill \textnormal{(sum rule)}.
    
    \item \( \cL_{fg}(A, B, E) = f(A) \cL_g(A, B, E) + \cL_f(A, B, E) g(B) \) \hfill \textnormal{(product rule)}.
    
    \item \( \cL_{f \circ g}(A, B, E) = \cL_f(g(A), g(B), \cL_g(A, B, E)) \), \hfill \textnormal{(chain rule)} \label{chain.rule}\\
    assuming further that the union of the spectra of \( g(A) \) and \( g(B) \) is strictly enclosed in \( \Gamma \).
\end{enumerate}
\end{lemma}

\begin{proof}
The first point follows from \cite[Prob.~3.9]{high:FM}. The second point follows directly from the formula \ref{Lab.operator}. For the other points, we use the formula \cite[Lem.~1.1]{kela98}, \cite[Thm.~4.12]{high:FM}

%\note{AA: cite Kenny and Laub as well}
\[
f\left(\begin{bmatrix}
         A  &  E \\ 0 & B
\end{bmatrix}\right)= f\left(\begin{bmatrix}
         A  &  0 \\ 0 & B
\end{bmatrix}\right)+L_f\left(\begin{bmatrix}
         A  &  0 \\ 0 & B
\end{bmatrix},\begin{bmatrix}
         0  &  E \\ 0 & 0
\end{bmatrix}\right).
\]
Aligning this formula with \eqref{block.def} yields
\begin{equation}\label{L.block.form}
  L_f\left(\begin{bmatrix}
         A  &  0 \\ 0 & B
\end{bmatrix},\begin{bmatrix}
         0  &  E \\ 0 & 0
\end{bmatrix}\right) = \begin{bmatrix}
         0  &  \cL_f(A,B,E) \\ 0 & 0
\end{bmatrix}.
\end{equation}
Therefore, the sum, product, and chain rules of the operator $\cL_f(A,B,\cdot)$ are directly inherited from those of the \FD\ operator \cite[Thm.'s~3.2--3.4]{high:FM}.
\end{proof}
These properties of the operator $\cL_f(A,B,\cdot)$ will simplify and facilitate our analysis. Similar to the work of Al-Mohy and Higham \cite{alhi09} on the \FD\ of the matrix exponential, we can extend a given numerical algorithm for a matrix function $f$ to an algorithm for simultaneously computing $f(A)$, $f(B)$, and $\cL_f(A,B,E)$. Al-Mohy and Higham \cite[Thm.~4.1]{alhi09} show that a given efficient evaluation scheme for a polynomial yields an efficient evaluation scheme for its \FD. We reach the same conclusion for $\cL_f(A,B,E)$ in view of Lemma \ref{lem.prop}. This will be demonstrated in the evaluation schemes for $\cL_{\exp}(A,B,E)$ in section \ref{sect.eval.exp}.

A significant application arises in exponential integrators, a class of numerical methods used for solving systems of ODEs, particularly those involving stiff or highly oscillatory problems:
\begin{eqnarray}\label{sys.odes}
    x@@'(t) = A@x(t)+h(t,x(t)), \quad x(t_0) = x_0,\quad t\ge t_0,
\end{eqnarray}
where $x(t)\in\C^n$, $A\in\C^{\nbyn}$, and $h$ is a nonlinear vector
function. Numerical solution schemes for such problems involve computing a linear combination of the form
\begin{equation}\label{phi.lin.comb}
\varphi_0(A)w_0+\varphi_1(A)w_1+\dots+\varphi_p(A)w_p,
\end{equation}
where
\[
\varphi_0(A)=\e^A,\quad \varphi_j(A)=A\varphi_{j+1}(A)+\frac{1}{j!},\quad j\ge0
\]
and $w_j$, $j=0\colon p$, are vectors that correspond to an integration method. For a comprehensive survey of the exponential integrators, see Hochbruck and Ostermann~\cite{hoos10}. In view of \cite[Thm.~2.1]{alhi11} and the block formula \eqref{block.def}, the linear combination \eqref{phi.lin.comb} satisfies
\[
\sum_{j=0}^{p}\varphi_j(A)w_j = \e^Aw_0 + 
\cL_{\exp}(A,J_p(0),W)e_p,
\]
where $J_p(0)$ is the Jordan block of size $p\times p$  associated with the zero eigenvalue, $W=[w_p,w_{p-1},\dots,w_1]\in\C^{n\times p}$, 
and $e_p$ denotes the
last column of the identity matrix $I_p$.

In control theory, Hamiltonian matrices frequently arise in the analysis of linear quadratic regulator (LQR) problems. Consider the Riccati matrix differential equation
\begin{equation*}
 X'(t)+X(t)A+A^TX(t)+X(t)RX(t)-C=0,
\end{equation*}
where $A$, $R$, and $C$ are constant matrices in $\R^{\nbyn}$, with $R$ and $C$ being symmetric positive semidefinite. This differential equation is closely related to the linear system~\cite{reid72}
\begin{equation*}
  \frac{\dt}{\dt t}\begin{bmatrix}
         u(t) \\ v(t)
\end{bmatrix} =\begin{bmatrix}
        A  &  R \\ C & -A^T
\end{bmatrix}\begin{bmatrix}
         u(t) \\ v(t)
\end{bmatrix}=:\At\begin{bmatrix}
         u(t) \\ v(t)
\end{bmatrix}.
\end{equation*}
The coefficient matrix $\At$ is \emph{Hamiltonian}, meaning that $J\At=(J\At)^T$, where $J=\antidiag(I_n,-I_n)$. The matrix exponential $\e^{t\At}$ arises in the solution of this system. A fundamental property of Hamiltonian matrices is the existence of
a structure-preserving unitary similarity transformation, known as the \emph{Real Schur-Hamiltonian decomposition} \cite[Thm.~5.1]{pavl81}, reducing $\At$ to a block upper triangular Hamiltonian matrix
\begin{equation}\label{tri.Ham}
 \Tt:=\begin{bmatrix}
        T  &  H \\ 0 & -T^T
\end{bmatrix},
\end{equation}
where $T$ is upper quasi-triangular with
eigenvalues in the left half plane and $H$ is symmetric. It follows that
\begin{equation*}
  \e^{\Tt} = \begin{bmatrix}
        F  &  D \\ 0 & F^{-T}
\end{bmatrix},
\end{equation*}
where $F=\e^T$ and $D=\cL_{\exp}(T,-T^T,H)$. Since $H$ is orthogonally diagonalizable,
$H=Q\Lambda Q^T$, the second property of Lemma~\ref{lem.prop} can be used to write
\[
D=Q@\cL_{\exp}(T,-T^T,\Lambda)@Q^T.
\]
Thus, Algorithm \ref{alg.expmF} below can be readily \emph{optimized} to compute $\cL_{\exp}(T,-T^T,\Lambda)$ efficiently.

Another interesting application of the exponential of block triangular matrices arises in option pricing models based on polynomial diffusions, a class of stochastic processes frequently used in mathematical finance to model the evolution of financial variables \cite{fila16}. In this context, Kressner, Luce, and Statti \cite{kls17} observe that exponentials of a sequence of recursively nested block upper triangular matrices need to be computed. Specifically, a sequence $\exp(G_0)$, $\exp(G_1)$, $\exp(G_2)$, $\dots$, where
\begin{equation}\label{G.seq}
  G_n=\begin{bmatrix}
        G_{n-1}  &  E_n \\ 0 & G_{n,n}
\end{bmatrix}, \quad G_0 = G_{0,0},
\end{equation}
with $G_{0,0}$ being an initial square matrix that serves as the starting point of the recursion. Defining the sequence $F_n=\exp(G_n)$ and $F_{n,n}=\exp(G_{n,n})$, the matrix 
$F_n$ can be computed recursively as
\begin{equation*}
  F_n=\begin{bmatrix}
        F_{n-1}  &  \cL_{\exp}(G_{n-1},G_{n,n},E_n) \\ 0 & F_{n,n}
\end{bmatrix}
\end{equation*}
using Algorithm \ref{alg.expmF}, which simultaneously computes all three blocks. 

Finally, a possible application that would enhance research in the computation of matrix functions is to use $\cL_f$ to compute the matrix function itself. Given an $\nbyn$ matrix, we can transform it to an upper triangular matrix $T$ using Schur decomposition. Thus for any index $1\le k\le n-1$, we have
\begin{equation}\label{part.k}
 f(T) = \begin{bmatrix}
         f(A_k)  &  \cL_f(A_k,B_k,E_k) \\ 0 & f(B_k)
\end{bmatrix},
\end{equation}
where $A_k=T(1\colon k,1\colon k)$, $B_k=T(k+1\colon n,k+1\colon n)$, and $E_k=T(1\colon k,k+1\colon n)$.

Although $\cL_f(A,B,E)$ can be computed directly by extracting the (1,2) block of the left-hand side matrix in \eqref{block.def}, increasing the size of the problem is computationally and storage-wise inefficient. We will discuss some disadvantages of this approach below. Therefore, having a specialized algorithm for evaluating the operator $\cL_f(A,B,E)$ is highly valuable.

The remainder of this paper is devoted to the computation of $\cL_{\exp}$.
%%%%%%%%%%%%%%%%%%%%%%%%%%%%%%%%%%%%%%%%%
\section{ Evaluation schemes for $\cL_{\exp}(A,B,E)$}
\label{sect.eval.exp}
The exponential function is entire over the complex plane, thus for any matrices $A$, $B$, and $E$ of appropriate matching sizes, the operator $\cL_{\exp}(A, B, \cdot)$ always exists, and the matrix $\cL_{\exp}(A, B, E)$ satisfies the Sylvester equation  
\[
AX - XB = \e^A E - E \e^B.
\] 
However, the converse is not true in general, since the Sylvester equation has a unique solution if and only if the spectra of $A$ and $B$ are disjoint. Even in this case, the accuracy of the computed solution depends critically on the separation between the spectra; that is, on how well-conditioned the problem is \cite[Ch.~9]{high:FM}. Due to this classical limitation, we adopt a different approach that imposes no restrictions on the matrices.

In this section we generalize the algorithm of Al-Mohy and Higham  \cite[Alg.~6.4]{alhi09} for \emph{simultaneously} computing the matrix exponential and
its \FD\ to an algorithm for simultaneously computing the matrix exponentials $\e^A$ and $\e^B$ and $\cL_{\exp}(A,B,E)$ using the scaling and squaring method based on the diagonal \Pant s, $r_m(z)=p_m(z)/q_m(z)$, for the exponential function $\e^z$. The cost analysis of Higham \cite{high05e,high09e}
suggests the optimal degrees for the \Pant s are the odd degrees
3, 5, 7, 9, and 13.
For $m=3, 5, 7, 9$, the polynomial
$p_m(z) = \sum_{i=0}^m c_i z^i$ can be written as
\begin{eqnarray}
     p_m(z) =
     z\sum_{k=0}^{(m-1)/2}c_{2k+1}z^{2k}+\sum_{k=0}^{(m-1)/2}c_{2k}z^{2k}
     =: u_m(z)+v_m(z).
   \label{p-odd-eval}
\end{eqnarray}
Therefore, $q_m(z) = -u_m(z)+v_m(z)$ because $p_m(-z) =
q_m(z)$ and
\[
    \cL_{p_m} =  \cL_{u_m} + \cL_{v_m}, \qquad
    \cL_{q_m} = -\cL_{u_m} + \cL_{v_m},
\]
where $\cL_{u_m}(A,B,E)$ and $\cL_{v_m}(A,B,E)$ are obtained by applying the sum and product rules of the operator on $u_m$ and $v_m$, \resp. This yields
\begin{eqnarray}
   \cL_{u_m}(A,B,E) &=& A\sum_{k=1}^{(m-1)/2}c_{2k+1}M_{2k}
                      +E\sum_{k=0}^{(m-1)/2}c_{2k+1}B^{2k}
     \label{sch_L_u}\\
   \cL_{v_m}(A,B,E) &=& \sum_{k=1}^{(m-1)/2}c_{2k}M_{2k}.
     \label{sch_L_v}
\end{eqnarray}
The matrix $M_\ell = \cL_{x^\ell}(A,B,E)$ is evaluated using the recurrence relation
\begin{equation}
%\label{M_rec1}
 M_\ell = A^tM_r+M_tB^r, \quad M_1=E,
\end{equation}
where $t$ and $r$ are positive integers such that $\ell=t+r$, resulting from applying the product rule to $z^\ell=z^tz^r$.

For $m=13$, the splitting of $p_{13}(z) = u_{13}(z)+v_{13}(z)$ into odd and even terms can be efficiently calculated \cite[sect.~6]{alhi09} as
\begin{eqnarray}
  &    u_{13}(z) = zw(z), \quad w(z)=z^6w_1(z)+w_2(z),
      \quad
      v_{13}(z) = z^6y_1(z)+y_2(z), &\nonumber\\
  &
  \begin{array}{cc} \label{m13.scheme}
  w_1(z) = c_{13}z^6+c_{11}z^4+c_9z^2,  &
  w_2(z) = c_7z^6+c_5z^4+c_3z^2+c_1,\\
  y_1(z) = c_{12}z^6+c_{10}z^4+c_8z^2, &
  y_2(z) = c_6z^6+c_4z^4+c_2z^2+c_0.
  \end{array}
  &
\end{eqnarray}
Applying the operator $\cL$ on theses schemes yields
\begin{eqnarray*}
   \cL_{u_{13}}(A,B,E) &=& A\cL_{w}(A,B,E) + Ew(B),\\
   \cL_{v_{13}}(A,B,E) &=& A^6\cL_{y_{1}}(A,B,E)+M_6y_1(B) + \cL_{y_{2}}(A,B,E),
\end{eqnarray*}
where
\begin{eqnarray*}
  \cL_{w}(A,B,E) &=& A^6\cL_{w_{1}}(A,B,E) + M_6w_1(B) + \cL_{w_{2}}(A,B,E),\\
  \cL_{w_{1}}(A,B,E) &=& c_{13}M_6+c_{11}M_4+c_9M_2,\\
  \cL_{w_{2}}(A,B,E) &=& c_7M_6+c_5M_4+c_3M_2,      \\
  \cL_{y_{1}}(A,B,E) &=& c_{12}M_6+c_{10}M_4+c_8M_2,\\
  \cL_{y_{2}}(A,B,E) &=& c_6M_6+c_4M_4+c_2M_2.
\end{eqnarray*}
Then
$\cL_{p_{13}}=\cL_{u_{13}}+\cL_{v_{13}}$ and
$\cL_{q_{13}}=-\cL_{u_{13}}+\cL_{v_{13}}$. We finally solve for
$r_m(A)$, $r_m(B)$, and $\cL_{r_m}(A,B,E)$ the equations
\begin{eqnarray}
     (-u_m+v_m)(A)r_m(A) &=& (u_m+v_m)(A),\nonumber% \label{rd}
     \\
     (-u_m+v_m)(B)r_m(B) &=& (u_m+v_m)(B),\nonumber% \label{rd}
     \\
     \qquad\quad
    (-u_m+v_m)(A)\cL_{r_m}(A,B,E) &=& \label{Lrd}
    (\cL_{u_m}+\cL_{v_m})(A,B,E)\\
     & +& @@(\cL_{u_m}-\cL_{v_m})(A,B,E)r_m(B).\nonumber
\end{eqnarray}
Observe that we can obtain $\cL_{r_m}(A,B,E)$ by solving the multiple left-hand
side linear system
\begin{eqnarray}
\cL_{r_m}(A,B,E)(-u_m+v_m)(B) &=& \label{L.H.S.Lrd}
    (\cL_{u_m}+\cL_{v_m})(A,B,E)\\
     &+& r_m(A)(\cL_{u_m}-\cL_{v_m})(A,B,E),\nonumber
\end{eqnarray}
which results from applying the product rule of the operator $\cL$ on the equation $r_m(z)q_m(z)=p_m(z)$ instead of $q_m(z)r_m(z)=p_m(z)$ that produces
\eqref{Lrd}. Thus, we can obtain $\cL_{r_m}(A,B,E)$ by solving the easiest system. As the exponential function is well-approximated by \Pant s near zero, if either $A$ or $B$ has a large spectral radius,
then the scaling and squaring method is used to reduce the spectral radii. This ensures that $r_m(2^{-s}A)^{2^s}$ and $r_m(2^{-s}B)^{2^s}$ are good approximations to $\e^A$ and $\e^B$, \resp, for suitably chosen scaling parameter $s$. Both $r_m(2^{-s}A)^{2^s}$ and $r_m(2^{-s}B)^{2^s}$ can be computed by repeated squaring. 

To approximate $\cL_{\exp}(A,B,E)$, we apply the operator $\cL$ to the approximation $r_m(2^{-s}z)^{2^s}\approx \e^z$ using the chain rule with $f(z)=z^{2^s}$ and $g(z)= r_m(2^{-s}z)$ in point \ref{chain.rule} of Lemma \ref{lem.prop}. Therefore,
\begin{equation}
\label{cL.approx}
\cL_{z^{2^s}}\bigl(X_0,Y_0,D_0\bigr)\approx
\cL_{\exp}(A,B,E),
\end{equation}
where
%\begin{equation*}
$X_0 = r_m(2^{-s}A)$, $Y_0 = r_m(2^{-s}B)$, and $D_0=\cL_{r_m}(2^{-s}A,2^{-s}B,2^{-s}E)$,
%\end{equation*}
which can all be computed using the evaluation scheme above. To recover the left hand side of \eqref{cL.approx} starting from the initial approximations $X_0$, $Y_0$, and $D_0$, let 
\[
 X_i = X_0^{2^i},\quad Y_i = Y_0^{2^i},\quad D_i=\cL_{z^{2^i}}\bigl(X_0,Y_0,D_0\bigr),\quad i=1\colon s. 
\]
Applying the product rule of the operator to the equation
$z^{2^i}= z^{2^{i-1}}z^{2^{i-1}}$ yields
\[
D_i = X_0^{2^{i-1}}D_{i-1} + D_{i-1}Y_0^{2^{i-1}}.
\]
This leads to the recurrence relation
\begin{equation}
\label{Li-exp}
\begin{aligned} 
  D_{i+1} &= X_i D_i + D_i @@Y_i, \\
  X_{i+1} &= X_i^2, \qquad Y_{i+1}@@=@@Y_i^2,\qquad   i=0\colon s-1,  
\end{aligned}
\end{equation}
from which $X_s\approx\e^A$, $Y_s\approx\e^B$, and $D_s\approx\cL_{\exp}(A,B,E)$.

In the next section, we conduct a backward error analysis for the problem and show how to select the scaling parameter $s$.

%%%%%%%%%%%%%%%%%%%%%%%%%%%%%%%%%
\section{Backward error analysis}
\label{sect.be.anal}
We analyze the backward error resulting from
the approximations $X_s\approx\e^A$, $Y_s\approx\e^B$, and $D_s\approx\cL_{\exp}(A,B,E)$ using the recurrence \eqref{Li-exp}. We recall the backward error result from~\cite[sect.~3]{alhi09a}.
%for
%the approximation of $\e^x$ by the \Pant, $r_m(x)=p_m(x)/q_m(x)$.
\begin{theorem}
Let
\begin{equation}\label{Omega-m}
  \Omega^{(n)}_m = \{\,Z\in\C^{\nbyn} : \rho( \e^{-Z} r_m(Z) - I ) < 1,\quad
  \rho(Z) < \mu_m \,\},
\end{equation}
where $\mu_m = \min\{\,|z| : q_m(z) = 0 \,\}$ and $\rho$ denotes the spectral radius. Then the functions
\begin{equation}
      g_{2m+1}(X) = \log( \e^{-X} @r_m(X) )
      \label{h2mp1}
\end{equation}
are defined for all $X\in\Omega^{(n)}_m$, where $\log$ is the
principal matrix logarithm, and 
\[
r_m(X)
= e^{X+g_{2m+1}(X)},\quad X\in\Omega^{(n)}_m.
\]
\end{theorem}
If $X$ is not in $\Omega^{(n)}_m$, choose $s$ so that $2^{-s}X \in
\Omega^{(n)}_m$. Then
\begin{equation}
r_m(2^{-s}X)^{2^s} = \e^{X+ 2^s g_{2m+1}(2^{-s}X)} =: \e^{X+\D X}
\label{be.eq}
\end{equation}
and the matrix $\D X = 2^s g_{2m+1}(2^{-s}X)$ represents the backward error associated with approximating $\e^X$ using the scaling and squaring method via \Pant s.
Over $\Omega^{(n)}_m$, the functions $g_{2m+1}$ have power series
expansions
\begin{equation}\label{power.h2m}
g_{2m+1}(X) = \sum_{k=0}^\infty b_{m,k}@@X^{2(m+k)+1}.
\end{equation}
Applying the operator $\cL$ to the sides of \eqref{be.eq} and using \eqref{cL.approx}, we obtain
\begin{eqnarray}
  \cL_{z^{2^s}}\bigl(X_0,Y_0,D_0\bigr) &=& \cL_{\exp\left(z+ 2^s g_{2m+1}(2^{-s}z)\right)}(A,B,E) \nonumber\\
  &=:& \cL_{\exp}(A+\D A,B+\D B,E+\D E),
\end{eqnarray}
where 
\[
\D A \!=\! 2^s g_{2m+1}(2^{-s}\!A),\quad \D B\! =\! 2^s g_{2m+1}\!(2^{-s}B),\quad
\D E\! = \cL_{g_{2m+1}}(2^{-s}\!A,2^{-s}B,E).
\]
These represent the backward errors corresponding to the approximations in \eqref{Li-exp}.
% Recall that $\D A$ and $\D B$ are the backward errors resulting from the %approximations of $\e^A$ and $\e^B$ by $r_m(2^{-s}A)^{2^s}$ and %$r_m(2^{-s}B)^{2^s}$, \resp, and they are independent. 
Since $A$ and $B$ are independent, let $X$ denote either $A$ or $B$.
Then the relative backward errors can be bounded~\cite{alhi09a} as
\begin{eqnarray}
% \nonumber % Remove numbering (before each equat`ion)
  \frac{ \norm{\D X} }{ \norm{X} } &=& \frac{ \norm{ g_{2m+1}(2^{-s}X)} }
                                              { \norm{2^{-s}X} } \nonumber
             \label{berr-bound2} \\%
                                   &\le& \frac{\Ht_{2m+1}(\norm{2^{-s}X})}{\norm{2^{-s}X}},
\end{eqnarray}
where $\Ht_{2m+1}(z) = \sum_{k=0}^\infty |b_{m,k}|@@z^{2(m+k)+1}$.

While $\D E$ depends on $A$, $B$, and $E$, the relative backward error $\norm{\D E}/\norm{E}$, for any subordinate matrix norm, can be bounded \emph{independently} of $E$ as proven in the following theorem.
\begin{theorem}
\label{Thm.DE}
Suppose that $A\in\C^{\nbyn}$, $B\in\C^{d\times d}$, and $s$ is a nonnegative integer such that $2^{-s}\!A\!\in\!\Omega^{(n)}_m$ and $2^{-s}B\in\Omega^{(d)}_m$. Then
the backward error $\D E\! =\! \cL_{g_{2m+1}}\!(2^{-s}\!A,2^{-s}B,E)$ satisfies the bound
\begin{equation}
\label{DE.bound}
\frac{ \norm{\D E} }{ \norm{E} } \le
%\sum_{k=0}^{\infty}\bigl(2(m+k)+1\bigr)|c_{2(m+k)+1}|@@
%                                %\left(2^{-s}\max(\norm{A},\norm{B})\right)^{2(m+k)}
\Htmd\bigl(2^{-s}\max(\norm{A},\norm{B})\bigr)
\end{equation}
for any subordinate matrix norm.
\end{theorem}
 \begin{proof}
The result follows from applying \cite[Thm.~3.2]{alhi09} on the equation in \eqref{L.block.form} for $f=g_{2m+1}$. Thus, we have
\begin{eqnarray*}
  \norm{\D E} &=& \left\|\begin{bmatrix}
         0  &  \cL_{g_{2m+1}}(2^{-s}A,2^{-s}B,E) \\ 0 & 0
\end{bmatrix}\right\| \\
              &=&\left\|L_{g_{2m+1}}\left(2^{-s}\begin{bmatrix}
         A  &  0 \\ 0 & B
\end{bmatrix},\begin{bmatrix}
         0  &  E \\ 0 & 0
\end{bmatrix}\right)\right\|\\
             &\le& \left\|\begin{bmatrix}
         0  &  E \\ 0 & 0
\end{bmatrix} \right\|@@ \Htmd\left( \left\|2^{-s}\begin{bmatrix}
         A  &  0 \\ 0 & B
\end{bmatrix} \right\|\right)\\
            &=& \norm{E}@@\Htmd\bigl(2^{-s}\max(\norm{A},\norm{B})\bigr).
\end{eqnarray*}
 \end{proof}

Define \cite[sect.~6]{alhi09}
\begin{equation}
 \theta_m = \max\{\,z : \Ht_{2m+1}(z)/z \le u \,\},\qquad
 \ell_m = \max\{\,z :
\Htmd(z) \le u\,\},
  \label{th-m2}
\end{equation}
 where $u  = 2^{-53} \approx 1.1\times
10^{-16}$ is the unit roundoff for IEEE double precision
arithmetic.
For $m=1\colon 20$, Al-Mohy and Higham evaluate $\theta_m$ and $\ell_m$ using high precision arithmetic and tabulate them in \cite[Table~6.1]{alhi09}. They observe that $\ell_m<\theta_m$ for all $m$. Thus, if the scaling parameter $s$ is chosen so that $2^{-s}\max(\norm{A},\norm{B})\le\ell_m$,
then the approximations of $\e^A$, $\e^B$, and $\cL_{\exp}(A,B,E)$, using the scaling and squaring method with diagonal \Pant s via the recurrence \eqref{Li-exp}, produce relative backward errors
$\norm{\D A}/\norm{A}$, $\norm{\D B}/\norm{B}$, and $\norm{\D E}/\norm{E}$
that never exceed $u$ in exact arithmetic.

The current version of the MATLAB function \expm\ implements the \alg\ of Al-Mohy and Higham \cite[Alg.~6.1]{alhi09a}; it has been in place since the release of version R2015b. The \alg\ uses the scaling and squaring method with \Pant s. The key features of the Al-Mohy and Higham \alg\ are twofold:
\begin{enumerate}
  \item The relative \be\ bound in \eqref{berr-bound2} is evaluated in terms of the sequence $\a_p(2^{-s}X)$ instead of $\norm{2^{-s}X}$, where
      \begin{equation}
        \a_p(X) = \max\bigl( \norm{X^{2p}}^{1/(2p)},\norm{X^{2p+2}}^{1/(2p+2)}
                    \bigr)
                    \label{def-alphap}
      \end{equation}
      and $p$ is chosen to minimize $\a_p(X)$ subject to $2m \ge 2p@(p-1)$ \cite[Eq.~(5.1b)]{alhi09a}.
      We have $\a_p(X)\le\norm{X}$, but $\a_p(X)$ can indeed be much smaller than $\norm{X}$ for highly nonnormal matrices. As a result, the \alg\ may select a smaller scaling parameter $s$.

  \item If the input matrix is (quasi)-triangular, as naturally obtained from the Schur decomposition, the \alg\ replaces the $2\times 2$ blocks in the squaring phase with the exponentials of their scaled diagonal counterpart blocks from the input matrix using exact formulas \cite[sect.~2]{alhi09a}.
\end{enumerate}
These two features increase the efficiency and accuracy of the scaling and squaring method, addressing a long-standing problem known as the \emph{overscaling phenomenon}. This problem was observed by Kenney and Laub for the matrix logarithm and the matrix exponential \cite{kela98}, and later investigated by Dieci and Papini \cite{dipa02}.

The question now is whether we can implement these points effectively. The answer is both yes and no. Unlike in \eqref{berr-bound2}, it remains unclear how to bound  the relative \be\ in \eqref{DE.bound} in terms of $\a_p(A)$ and $\a_p(B)$.
However, we can still benefit from the idea in point 2 above by applying Schur decomposition to transform the matrices $A$ and $B$ into (quasi-)triangular forms. Subsequently, we undo the Schur decomposition using the second point of Lemma \ref{lem.prop}.

We are now ready to describe and formulate our algorithm.
 %%%%%%%%%%%%%%%%%%%%%%%%%%%%%%%%%%%%%%%%
\section{New and existing algorithms}
\label{sect.algs}
In this section, we present our algorithm for simultaneously computing the matrices $\e^A$, $\e^B$, and $\cL_{\exp}(A, B, E)$ using the scaling and squaring method, where the selection of the scaling parameter is based on our backward error analysis. First, we describe and present our algorithm; then, we review the existing algorithm of Kenney and Laub \cite{kela98}.
\subsection{New algorithm}
There are several points and techniques we would like to emphasize.

First, in view of \eqref{DE.bound}, the norm of the matrix $E$ plays no role in the selection of the scaling parameter $s$ when computing $\e^A$, $\e^B$, and $\cL_{\exp}(A, B, E)$ via the scaling and squaring method. This is in contrast to evaluating them using the block matrix as in \eqref{block.def}, where the scaling parameter depends on the norm of the entire block matrix. Consequently, if $\norm{E} \gg \max(\norm{A}, \norm{B})$, a larger value of the scaling parameter may be selected. This can lead to overscaling the diagonal block matrices, affecting the accuracy of the scaling and squaring method.

Second, even though $E$ has no contribution in the selection of the scaling parameter, overscaling one component of the diagonal block matrices is still possible. For instance, if $\norm{A}$ is large and $\norm{A} \gg \norm{B}$, then $B$ would be overscaled. Thus, computing $\e^B$ by $s$ times repeated squaring of $Y_0 = r_m(2^{-s}B)$ in \eqref{Li-exp} could result in a loss of accuracy. To address this issue, we suggest using the Schur decomposition in a specific setting described below.
\begin{algorithm}
\caption{Scaling and Squaring Algorithm for the Exponential of a Block  
Triangular Matrix. \emph{For $A\in\C^{n\times n}$, $B\in\C^{d\times d}$, and $E\in\C^{n\times d}$, this algorithm computes $X\approx\e^A$, $Y\approx\e^B$, and
$D\approx\cL_{\exp}(A,B,E)$ simultaneously by the scaling and squaring method with Pad\'e approximants.
The parameters $\ell_m$ are listed in \textnormal{\cite[Table~6.1]{alhi09}}.
The algorithm is intended for IEEE double precision arithmetic.}}
\label{alg.expmF}
\begin{algorithmic}[1]
%\Statex \textbf{Input:} $A\in\C^{n\times n}$, $B\in\C^{d\times d}$, %$E\in\C^{n\times d}$
%\Statex \textbf{Output:} $X=\e^A$, $Y=\e^B$, $D=\cL_{\exp}(A,B,E)$
%\Statex \textbf{Parameters:} $\ell_m$ from \cite[Table~6.1]{alhi09} (for IEEE %double precision)
%
\State $\eta = \max(\normi{A},\normi{B})$
\For{$m \in \{3,5,7,9\}$}
    \If{$\eta \le \ell_m$}
        \State Compute $U_a=u_m(A)$, $V_a=v_m(A)$, $U_b=u_m(B)$, $V_b=v_m(B)$ via \eqref{p-odd-eval}
        \State Compute $D_u=\cL_{u_m}(A,B,E)$, $D_v=\cL_{v_m}(A,B,E)$ via \eqref{sch_L_u}, \eqref{sch_L_v}
        \State $s =0$, \textbf{goto} \ref{line.mUVa}.
    \EndIf
\EndFor

\State $s = \lceil \log_2(\eta/\ell_{13})\rceil$ \Comment{Smallest integer with $2^{-s}\eta\le\ell_{13}$}
%\If{$s \ge 10$}
\State If $s\ge10$, compute Schur decomposition for $A$ and $B$; update $A$, $B$, and $E$.
%\EndIf

\State $A \gets 2^{-s}A$, $B \gets 2^{-s}B$, $E \gets 2^{-s}E$

\State $A_2 = A^2$, $A_4 = A_2^2$, $A_6 = A_2A_4$
\State $B_2 = B^2$, $B_4 = B_2^2$, $B_6 = B_2B_4$

\State $M_2 = AE+EB$, $M_4 = A_2M_2+M_2B_2$, $M_6 = A_4M_2+M_4B_2$

\State $W_{1a} = w_1(A)$, $W_{1b} = w_1(B)$ \Comment{using \eqref{m13.scheme}}
\State $W_{2a} = w_2(A)$, $W_{2b} = w_2(B)$
\State $Z_{1a} = y_1(A)$, $Z_{1b} = y_1(B)$
\State $Z_{2a} = y_2(A)$, $Z_{2b} = y_2(B)$

\State $W_a = A_6W_{1a}+W_{2a}$, $W_b = B_6W_{1b}+W_{2b}$
\State $U_a = AW_a$, $U_b = BW_b$
\State $V_a = A_6Z_{1a}+Z_{2a}$, $V_b = B_6Z_{1b}+Z_{2b}$

\State $D_{w_1} = c_{13}M_6+c_{11}M_4+c_9M_2$
\State $D_{w_2} = c_7M_6+c_5M_4+c_3M_2$
\State $D_{y_1} = c_{12}M_6+c_{10}M_4+c_8M_2$
\State $D_{y_2} = c_6M_6+c_4M_4+c_2M_2$
\State $D_w     = A_6D_{w_1}+M_6W_{1b}+D_{w_2}$
\State $D_u     = AD_w+EW_b$
\State $D_v     = A_6D_{y_1}+M_6Z_{1b}+D_{y_2}$
\State Solve $(-U_a+V_a)X = U_a+V_a$ for $X$
\label{line.mUVa} 
\State Solve $(-U_b+V_b)Y = U_b+V_b$ for $Y$
\label{line.mUVb} 

\State \parbox[t]{\dimexpr\linewidth-\algorithmicindent}{
\vspace{-5.0ex}
\begin{minipage}[t]{\linewidth}
\begin{flalign*}
\text{Solve for } D &\text{ either}\hspace{-5.0em} & (-U_a+V_a)D &= D_u+D_v+(D_u-D_v)Y, &\\
                     &\text{ or}\hspace{-5.0em}    & D(-U_b+V_b) &= D_u+D_v+X(D_u-D_v) &
\end{flalign*}
\end{minipage}}
\label{line.mUVL}
\For{$k = 1$ to $s$}
    \State $D \gets XD + DY$
    \If{$X$ or $Y$ is (quasi-)triangular}
        \State Apply \cite[Code Fragments~2.1--2.2]{alhi09a}
    \Else
        \State $X \gets X^2$, $Y \gets Y^2$
    \EndIf
\EndFor
\end{algorithmic}
\end{algorithm}
%%%%%%%%%%%%%%%%%%%%%%%%%%%%%%%%%

Algorithm~\ref{alg.expmF} is transformation-free in nature. However, we may use Schur decomposition to reduce the computational cost and enhance the stability of the squaring phase. For simplicity, we assume that $A$, $B$, and $E$ are full square matrices of dimension $n$. The algorithm involves $25 + 4s$ matrix multiplications and two LU factorizations of the coefficient matrices in lines~\ref{line.mUVa} and~\ref{line.mUVb}, used to solve four triangular systems in those lines, plus two additional triangular solves in line~\ref{line.mUVL}. We use the operation counts in flops as tabulated by Higham \cite[App.C]{high:FM}.
The total operation count is about $(172 + 24s)n^3/3$ flops. If we reduce $A$ and $B$ to upper triangular forms using Schur decomposition, which typically costs $25n^3$ flops for each, while $E$ is full, the total operation count is about $(248 + 8s)n^3/3$ flops, which is greater than the operation count of the transformation-free form for $0 \le s \le 4$. As $s$ increases, the operation count suggests the use of Schur decomposition, allowing the algorithm to exploit the (quasi-)triangularity to enhance the stability of the squaring phase. If $A$ and $B$ are real matrices, we compute the real Schur decomposition to keep the operations in real arithmetic.
Thus, for stability and efficiency reasons, we design our algorithm to perform Schur decomposition if the scaling parameter $s$ exceeds 9; that is, if $\eta \ge 2^{10}\ell_{13} \approx 4.85\times10^3$.

If we use Algorithm~\ref{alg.expmF} to compute the matrix exponential for triangular matrices as proposed in \eqref{part.k}, the algorithm will be more efficient if we choose $k$ that minimizes $\max(\norm{A_k}, \norm{B_k})$.
%%%%%%%%%%%%%%%%%%%%%%%%%%%%%%%
\subsection{Existing algorithm}
Kenney and Laub \cite[sect.~3.3]{kela98} develop an algorithm, 
known as the Schur--Fr\'echet algorithm for computing the exponential. This algorithm is specifically designed for upper triangular matrices partitioned as in \eqref{block.def}, where the exponentials of the diagonal blocks are computed independently, and the off-diagonal blocks are computed by exploiting the integral formula \cite[Eq.~(10.40)]{high:FM} associated with the operator
$\cL_{\exp}(A,B,E)$:
\begin{equation}\label{int.cL}
\cL_{\exp}(A,B,E)=\int_{0}^{1} \e^{tA}E@@\e^{(1-t)B}\dt t.
\end{equation}
Their approach approximates the integral by using the Kronecker sum representation of the operator \cite[sect.~2.2]{kela98}, under the assumption that $A=B$. By analogy with their analysis and \cite[Thm.~10.13]{high:FM}, applying the $\vec$ operator to the sides of \eqref{int.cL} yields 
\begin{align}
  \vec\bigl(\cL_{\exp}(A,B,E)\bigr) &= \nonumber
  \int_{0}^{1}\left(\e^{tA^T}\otimes\e^{(1-t)B}\right)\vec(E)\dt t\\
     &= \frac{1}{2}\bigl(\e^{A^T}\oplus\e^{B}\bigr)
     \tau\bigl(\frac{1}{2}[A^T\oplus(-B)]\bigr)\vec(E), \label{tau.result}
\end{align}
where $\otimes$ and $\oplus$ denote the Kronecker product and sum, \resp, 
\cite[App.~B13]{high:FM} and
$\tau(z) = \tanh(z)/z$. The integral result \eqref{tau.result} is valid if $\norm{A^T\oplus(-B)}_2<\pi$, which in turn holds if the value $\eta:=\max(\normF{A},\normF{B})<\pi/2$ \cite[Thm.~2.1]{kela98}. Using the [8/8] \Pant, $r_8(z)$, to $\tau(z)$ factored
in the form
\begin{equation*}
  \tau(z)\approx r_8(z)=\prod_{j=1}^{8}\frac{1-z/\alpha_j}{1-z/\beta_j},
\end{equation*}
they evaluate $R_8$ such that $\vec(R_8)=r_8\left(\frac{1}{2}[A^T\oplus(-B)]\right)\vec(E)$
by solving the following sequence of Sylvester equations:
\begin{align}
  R_0 &= E \nonumber\\
  \left(I_n+\frac{A}{\beta_j}\right)R_j+R_j\left(I_d-\frac{B}{\beta_j}\right)
   &=
   \left(I_n+\frac{A}{\alpha_j}\right)R_{j-1}+R_{j-1}\left(I_d-\frac{B}{\alpha_j}\right) \label{sylv.cascade}
\end{align}
for $j=1\colon8$. Therefore,
\begin{equation}\label{cL.r8.approx}
\cL_{\exp}(A,B,E)\approx \frac{1}{2}(\e^AR_8+R_8\e^B).
\end{equation}
Due to the restriction $\norm{A^T\oplus(-B)}_2<\pi$ because $\tanh(z)$ has poles with
$|z|=\pi/2$ and the fact that the quality of the approximation \eqref{cL.r8.approx} depends on the quality of the approximation of
$\tau(z)$ by $r_8(z)$, Kenney and Laub propose scaling $A\gets 2^{-s}A$ and $B\gets 2^{-s}B$ so that $2^{-s}\eta\le 1/4$ to control the \cn\ of the Sylvester equations above and the following forward error
%
% \footnote{For double precision arithmetic, it suffices to have %$2^{-s}\eta\le1$, for which $g(1)\approx 1.44\times10^{-16}$ as revised by %Higham \cite[pp.~257]{high:FM}.}.
\begin{equation*}
\norm{\tau(2^{-s}X)-r_8(2^{-s}X)}_2\le g(\norm{2^{-s}X}_2)\le g(\eta)\le g(1/4)
\approx 6.85\times10^{-28},
\end{equation*}
where $X=\frac{1}{2}[A^T\oplus(-B)]$ and $g(z)=\tau(\1i z)-r_8(\1i z)$ by \cite[Thm.~2.3]{kela98}. If $s=0$, then the result follows from \eqref{cL.r8.approx}; otherwise, the algorithm invokes the ``top down square root Sylvester cascade":
\begin{algorithmic}[1]
\State Compute $D = R_8$ using \eqref{sylv.cascade} with $A \gets 2^{-s}A$ and $B \gets 2^{-s}B$
\State $X = \e^{A/2}$, $Y = \e^{B/2}$ \label{line1}
\Comment{unscaled $A$ and $B$}
\For{$j = 1$ \textbf{to} $s$}
    \State $D \gets XD + DY$
    \If{$j < s$}
        \State $X \gets \e^{A/2^{j+1}}$, $Y \gets \e^{B/2^{j+1}}$ \label{lin.rf} \Comment{recursive form~~~~~~~}
        \State \textbf{or}
        \State $X \gets X^{1/2}$, $Y \gets Y^{1/2}$ \label{lin.srf}
        \Comment{semirecursive form }        
    \EndIf
\EndFor
\State $D \gets XD + DY$
\State $\cL_{\exp}(A,B,E) \approx D/2^{s+1}$
\end{algorithmic}
The authors propose two forms of their Schur--Fr\'echet algorithm: the ``recursive form", where the algorithm invokes separate calculations of the
subsequent exponentials $\e^{A/2^{j+1}}$ and $\e^{B/2^{j+1}}$ in line \ref{lin.rf} above, or the ``semirecursive form", in which the exponentials in step $j+1$ are treated as square roots of the exponentials computed in step $j$.
The reason of the recursive form is to mitigate the overscaling problem by computing the exponentials independently at each iteration. Recall that $s$ is the scaling parameter for the approximation of $\tau(z)$ by $r_8(z)$; it is not that used by the scaling and squaring method for the matrix exponential.
Therefore, any algorithm for the matrix exponential can be applied here, such as the general-purpose MATLAB function \funm\ or the well-known function \expm. The scaling parameter $s$ is independent of the scaling parameters selected by \expm\ for $A$ and $B$. For example, the matrices
 \[
 A = \begin{bmatrix}
         -2\times 10^3  &  10^7 \\ 0 & 3
\end{bmatrix},\qquad
B = \begin{bmatrix}
         0  &  2 \\ 1 & 1
\end{bmatrix}
\]
 have $\eta \approx 10^7$ and hence $s=26$ whereas the \expm\ scaling parameters
 for $A$ and $B$ in line \ref{line1} are 10 and 0, \resp.
 
The recursive form has several advantages. First, the independent evaluation of the exponentials in each iteration reduces the risk of overscaling. Second, it takes advantage of the latest developments in matrix exponential algorithms. However, the independent evaluation of the exponentials, as in line \ref{lin.rf}, significantly increases the cost of the algorithm. To address this, the authors proposed the semirecursive version to reduce computational costs especially when the matrices are triangular. However, they observe some instability in the repeated square root process due to underflow during the initial exponentiation in line \ref{line1}. In the example above, the $(1,1)$ element of $X=\e^{A/2}$ underflows to zero, making it irrecoverable when taking square roots of $X$. To overcome this problem, Kenney and Laub heuristically suggest not reverting from the recursive version to the semirecursive version (in line \ref{lin.srf}) unless $\eta/2^k\le 100$, where $1\le k<s$. The algorithm always invokes complex arithmetic even though all the input matrices are real. The reason is that the coefficients $\a_j$ and $\beta_j$ in \eqref{sylv.cascade} are complex numbers. In contrast, our algorithm's outputs and all subsequent operations are performed in real arithmetic, provided that all input matrices are real.

In fact, there is a major problem with the idea of the semirecursive form. We will first quote the authors and then comment on it. They say \cite[pp.~656]{kela98}:
\begin{quote}
 ``As a compromise between accuracy and efficiency we have modified the above algorithm to a ``semirecursive'' form: if $X_{ii}$ [scaled $A$ and $B$ in our notation] is not too large in norm ... replace the exponentiation $A_{ii}\gets\e^{X_{ii}/2}$ by the analytically equivalent square root operation $A_{ii}\gets A_{ii}^{1/2}$..., we replaced the exponentiation with the square root if the condition $\normF{X_{ii}} < \tol$ with $\tol=100$ was satisfied.''
 \end{quote}
This idea does not generally apply to all matrices. The identity $\e^{X/2}=(\e^{X})^{1/2}$ is guaranteed to hold if the imaginary parts of the eigenvalues of the matrix $X$ lie within the interval $(-\pi,\pi]$. Therefore, the semirecursive form has limited practical usefulness.
%%%%%%%%%%%%%%%%%%%%%%%%%%%%%%%%%%%%%%%%%%%%%%%%%%%%%%%
\section{Numerical experiment}
\label{sect.num}
This experiment consists of three parts for computing $\cL_{\exp}(A, B, E)$. All experiments are performed in MATLAB R2022b, and the algorithms are implemented as follows.
\begin{enumerate}
  \item Algorithm \ref{alg.expmF} proceeds without transformations unless the scaling parameter satisfies $s \ge 10$, in which case the Schur decompositions of $A$ and $B$ are computed.
  \item The algorithm of Kenney and Laub (KL) is implemented in its best performing setting: the recursive form with initial Schur decompositions of $A$ and $B$. The Sylvester equations \eqref{sylv.cascade} are solved using the MATLAB function \sylvester, and the matrix exponentials are computed using \expm.
  \item \expm\ is applied directly to the block matrix in \eqref{block.def}, and the $(1,2)$ block of the result is extracted.
  \item The ``exact'' value of $\cL_{\exp}(A, B, E)$, denoted by $D^{\refsol}$, is computed with 300-digit precision arithmetic using the algorithm of Fasi and Higham \cite{fahi19} applied to the block matrix in \eqref{block.def}, followed by extraction of the $(1,2)$ block. Their MATLAB implementation, \expmmp, utilizes the Advanpix Multiprecision Computing Toolbox \cite{adva-mct}.
\end{enumerate}

First, we construct a block upper triangular Hamiltonian matrix of the form \eqref{tri.Ham} to emulate a practical application scenario. We take $T$ as the Schur triangular factor obtained from the MATLAB function \t{gallery('invol',8)} and then negate its positive diagonal elements to ensure that $T$ is stable. We set $H=(C+C^T)/2$, where $C$ is \t{gallery('chebspec',8)}. Thus we consider the matrix
\begin{equation*}
\Tt_k=\begin{bmatrix}
        T  &  \a_kH \\ 0 & -T^T
\end{bmatrix},\quad \a_k = 2^{t_k},\quad t_k=200(k-3),\quad k=0\colon6.
\end{equation*}
We then compute the $(1,2)$ block of $\e^{T_k}$, which is 
$\cL_{\exp}(T, -T^T,\a_k H)=:D^{\refsol}(\a_k)$, by the four algorithms described above. Denoting the computed result by $\Dhat(\a_k)$, we evaluate the 2-norm relative error $\normt{D^{\refsol}(\a_k)-\Dhat(\a_k)}/\normt{D^{\refsol}(\a_k)}$ for each $k$.

The results in Table~\ref{table.err} clearly illustrate the comparative numerical stability and accuracy of the tested algorithms under varying levels of nonnormality introduced by the parameter $\alpha_k=2^{t_k}$.
For all $\alpha_k$, the relative errors produced by Algorithm \ref{alg.expmF} and the KL algorithm are totally unaffected by $\alpha_k$, though Algorithm \ref{alg.expmF} delivers highly accurate relative errors close to machine precision. In contrast, the KL algorithm exhibits relative errors of the order of $10^{-10}$.

The direct application of \expm\ to the block matrix yields satisfactory results for small values of $\alpha_k$, but its accuracy declines as $k$ increases. For the largest value tested, $\alpha_6 = 2^{600}$, the relative error rises rapidly to approximately $9.39 \times 10^{-2}$. This reduction in accuracy is primarily due to the $(1,2)$ block norm becoming arbitrarily large, leading to excessive scaling that overscales the diagonal blocks and adversely affects numerical stability. While this represents an extremely exaggerated scenario unlikely to arise in practice, \expm\ still outperforms the KL algorithm up to $k = 5$. Nevertheless, this behavior highlights the importance of Algorithm~\ref{alg.expmF} as a specialized and reliable method for this class of problems.

Furthermore, both Algorithm~\ref{alg.expmF} and the KL algorithm \emph{preserve the linearity} of the operator $\cL_{\exp}(A, B, E)$. This property holds exactly in exact arithmetic, as it is inherent in the derivation of both algorithms.
In finite precision arithmetic, linearity is also effectively preserved under the current experimental conditions. Notably, the parameter $\alpha_k$ is selected as an \emph{integer power} of 2, ensuring that the scaling operation $\alpha_k H$ introduces no rounding errors. This is a direct consequence of the fact that scaling by integer powers of 2 in IEEE floating-point arithmetic is exact.
Thus, we have $\Dhat(\alpha_k) = \alpha_k \Dhat(1)$, and the factor $\alpha_k$ cancels out in the relative error computation.

The algorithmic structure of Algorithm~\ref{alg.expmF} exploits the triangularity of the input matrices and selects the scaling parameter independently of the matrix $\alpha_k H$. Although the KL algorithm similarly does not rely on the norm of $\alpha_k H$ for scaling, its loss of accuracy is attributed to the nature of its algorithmic design rather than to a loss of linearity.
%%%%%%%%%%%%%
 \captionsetup[table]{skip=5pt, font=small}
 
 \begin{table}[htbp]
\centering
\caption{Relative errors in computing the (1,2) block of $\e^{T_k}$ for 
 $\a_k=2^{t_k}$, $t_k=200(k-3)$, and $k=0\colon6$. }
\label{table.err}
\begin{tabular}{
    S[table-format=-4.0]                % Value column (integers)
    S[table-format=1.3e-2, round-mode=figures, round-precision=4]  % Metric A
    S[table-format=1.3e-2, round-mode=figures, round-precision=4]  % Metric B
    S[table-format=1.3e-2, round-mode=figures, round-precision=4]  % Metric C
}
\toprule
{$t_k$} & {Alg. 4.1} & {KL alg.} & {\expm} \\
\midrule
-600 & 9.916e-16 & 9.542e-09 & 2.968e-15 \\
-400 & 9.916e-16 & 9.542e-09 & 2.968e-15 \\
-200 & 9.916e-16 & 9.542e-09 & 2.968e-15 \\
   0 & 9.916e-16 & 9.542e-09 & 2.968e-15 \\
 200 & 9.916e-16 & 9.542e-09 & 7.634e-09 \\
 400 & 9.916e-16 & 9.542e-09 & 1.990e-09 \\
 600 & 9.916e-16 & 9.542e-09 & 9.391e-02 \\
\bottomrule
\end{tabular}
\end{table}

Second, we used the set of test matrices employed by Al-Mohy, Higham, and Liu \cite{ahl22}. This collection consists of nonnormal matrices selected from the Anymatrix collection of Higham and Mikaitis \cite{himi-am}, as well as from the matrix functions literature \cite{alhi09, ahr15, fahi19}. The collection is available on GitHub, maintained by Liu, at \url{https://github.com/Xiaobo-Liu/mp-cosm}. The function \testmats\ in that repository provides 99 test matrices.

We randomly selected matrices $A$ and $B$ from this set with dimensions 30 and 20, respectively, and paired them. Random permutations of the index set ${1, 2, \dots, 99}$ were generated using the MATLAB built-in function \t{randperm(99)}. A different random matrix $E$ was generated for each pair of $A$ and $B$ using \randn. For each triplet $(A, B, E)$, we computed 
$\cL_{\exp}(A, B, E)$, denoted by $\Dhat$, using Algorithm~\ref{alg.expmF}, KL algorithm, and \expm\ as described above.
% for $\a = 1$, and using only Algorithm~\ref{alg.expmF} and the KL algorithm for %$\a = 1 / \normi{E}$; the latter specifically caters to testing linearity preservation. }
%%%%%%%%%%%%%%%%%%%%%%%%%%%%%%%%%%%%%%%%%%%%%%%%%%%%%
\begin{figure}
  \centering
  \myfig{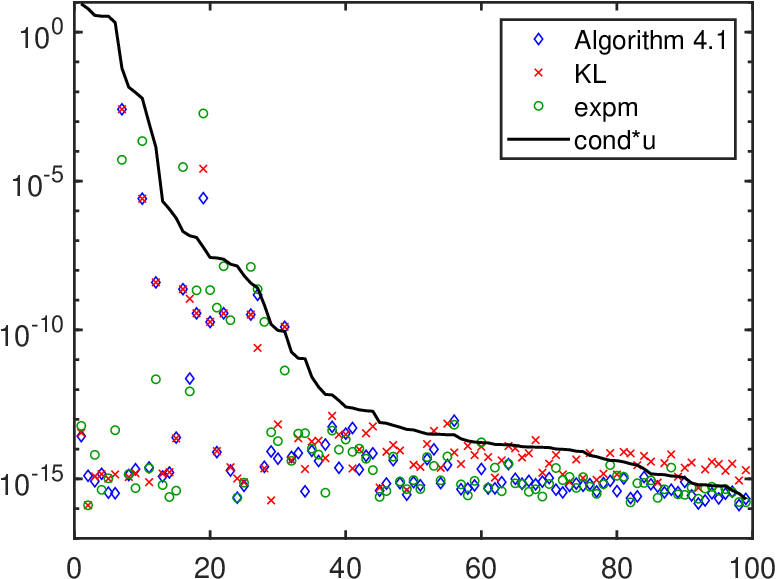}
  \caption{Relative forward errors for the computation of $\cL_{\exp}(A,B,E)$
  by Algorithm \ref{alg.expmF}, Kenney-Laub's algorithm (KL), and \expm. The solid
  line represents the \cn, $\cond_{\cL_{\exp}}(A,B,E)$, multiplied by $u=2^{-53}$.}
  \label{fig.test1}
\end{figure}
%
%%%%%%%%%%%%%%%%%%%%%%%%%
\begin{figure}
  \centering
  \myfig{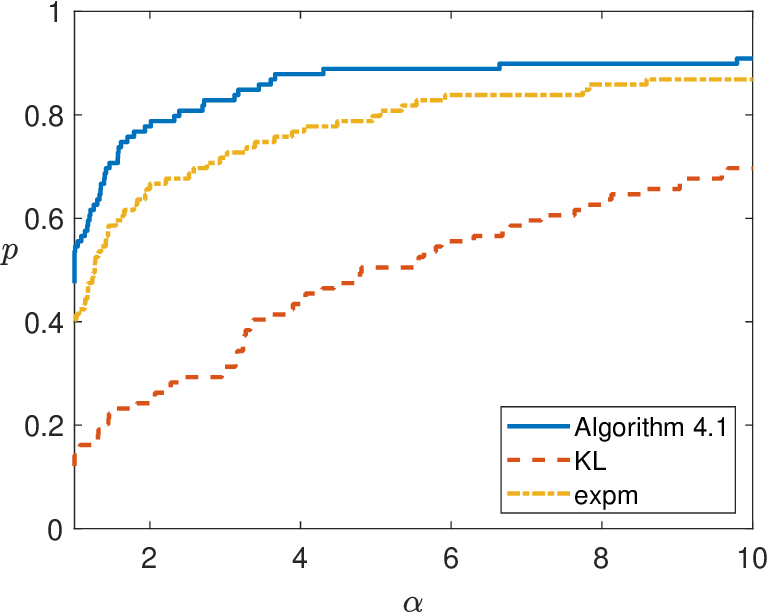}
  \caption{The data in Figure \ref{fig.test1} presented as a performance profile. }
  \label{fig.test2}
\end{figure}
%\cite{himi21}
%%%%%%%%%%%%%%%%%%%%%%%%%
\begin{figure}
  \centering
  \myfig{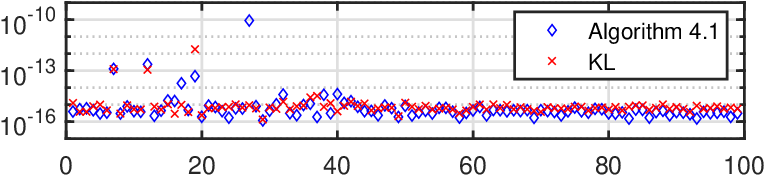}
  \caption{Relative errors for $\Dhat(\a) \approx \a \Dhat(1)$ with $\a = 1/\normi{E}$ using Algorithm~\ref{alg.expmF} and the KL algorithm. Data sorted as in Figure~\ref{fig.test1}.
 }
  \label{fig.test3}
\end{figure}
%%%%%%%%%%%%%%%%%%%%%%%%%%%%%%%%%%%%%%%%%%%%%%%%%%%%%%%

Figure \ref{fig.test1} shows the relative forward errors $\normi{D^{\refsol}-\Dhat}/\normi{D^{\refsol}}$ for each algorithm, sorted in descending order according to the condition number of the problem, $\cond_{\cL_{\exp}}$, defined below.
The solid black line represents the condition number multiplied by the unit roundoff. The condition number measures the forward stability of these algorithms; the algorithm is considered numerically stable if the relative forward error does not exceed $\cond_{\cL_{\exp}}u$.
We define the condition number as
\begin{align*}
\cond_{\cL_{\exp}}(A,B,E)\! =& \lim_{\eps\to 0}\!\sup_{\substack{\norm{\D A} \le \eps \norm{A}
 \\ \norm{\D B} \le \eps \norm{B} \\ \norm{\D E} \le \eps \norm{E}} }
\!\!\!\!\!\!\!\!\!\frac{\norm{\cL_{\exp}(A + \D A, B+\D B,E + \Delta E) - \cL_{\exp}(A,B, E)}}
{\eps\norm{\cL_{\exp}(A,B,E)}}.
\end{align*}
Exploiting the relation \eqref{L.block.form} between the operator $\cL_f$ and the block version of the \FD\ with some manipulation, we take as a condition number
\begin{equation*}
 \cond_{\cL_f}(A,B,E) = \cond_{L_f}\left(\begin{bmatrix}
         A  &  0 \\ 0 & B \end{bmatrix},
         \begin{bmatrix}
         0  &  E \\ 0 & 0 \end{bmatrix}\right),
\end{equation*}
where $\cond_{L_f}(X,Z)$ is the condition number of the \FD\ of the matrix function $f$ at $X$ in direction $Z$. Higham and Relton \cite{hire14b} propose an
algorithm estimating $\cond_{L_f}(X,Z)$ in the 1-norm. Their algorithm subsequently requires algorithms for computing the \FD\ $L_f(A,E)$ and estimating the \cn\ of the matrix function. We provide the algorithm of Higham and Relton with \cite[Algs.~6.4 \& 7.4]{alhi09} for $f=\exp$.  

Figure~\ref{fig.test1} clearly shows that all the algorithms exhibit forward stability, with Algorithm~\ref{alg.expmF} demonstrating the highest level of stability, followed closely by \expm\ applied to the $2 \times 2$ block matrix. More precisely, the forward relative errors produced by Algorithm~\ref{alg.expmF} lie below the solid line representing $\cond_{\cL_{\exp}} u$ for 95 problems, followed by 93 problems for \expm\ and 77 problems for the KL algorithm.

To better highlight the differences between the methods, Figure \ref{fig.test2} presents the data from Figure \ref{fig.test1} in the form of a performance profile \cite{dihi13,domo02}.
For a given $\a\geq 1$, each point on the curve represents the fraction $p$ of problems for which the method’s error is within a factor $\alpha$ of the best error achieved by any method.

The plot reveals several important insights. Algorithm~\ref{alg.expmF} exhibits the most favorable profile, with its curve rising rapidly and plateauing near $p = 0.95$, indicating that it achieves near-optimal accuracy on approximately 95\% of the problems.

\expm\ also performs competitively but plateaus at a lower level (around $p = 0.85$), suggesting that while it can deliver accurate results, it does not consistently match the reliability of Algorithm~\ref{alg.expmF}. Notably, \expm\ is unable to fully match the success rate of Algorithm~\ref{alg.expmF}, even when the error tolerance is relaxed to very large values of $\alpha$.

 The KL algorithm shows the slowest ascent, plateauing below $p = 0.7$. Matching the success rate of Algorithm~\ref{alg.expmF} requires a relaxation factor as large as $\alpha \approx 46$, highlighting its limited competitiveness in both accuracy and robustness. Table~\ref{table.pp} summarizes success rates at selected $\alpha$ values and the minimum $\alpha$ needed for each method to match Algorithm~\ref{alg.expmF}. A value of $\infty$ indicates the method never reaches this success rate.
 
Finally, concerning linearity preservation for scaling factors that are not integer powers of 2, we denote by $\Dhat(\a)$ the computed values of $\cL_{\exp}(A, B, \a E)$. Both $\Dhat(\a)$ and $\Dhat(1)$ are computed independently across the test matrices using Algorithm~\ref{alg.expmF}, and similarly using the KL algorithm. This setup verifies whether the computed $\Dhat(\a)$ satisfies the linearity relation $\Dhat(\a) = \a \Dhat(1)$, with $\a = 1 / \normi{E}$, a scaling factor not restricted to integer powers of 2.
Although this relation holds exactly in exact arithmetic, scaling by arbitrary real numbers introduces rounding errors in finite precision, which may affect its preservation. Figure~\ref{fig.test3} shows that both algorithms maintain linearity with high accuracy, as relative errors remain close to machine precision in most cases. A few larger errors appear but remain within acceptable range.
For completeness, we repeated the experiment with $\a$ restricted to integer powers of 2 and found that the relation $\Dhat(\a) = \a \Dhat(1)$ holds \emph{identically}, confirming that rounding errors are entirely avoided in this case.
\begin{table}[htbp]
\centering
\caption{Percentage of problems with error within a factor $\alpha$ of the best. The last column gives the minimum $\alpha$ needed to match Algorithm~\ref{alg.expmF}’s success rate.}

\label{table.pp}
\small  % Or \scriptsize if still too wide
\begin{tabular}{lccc}
\toprule
Method &  $\alpha=2$ (\%) & $\alpha=5$ (\%) & $\alpha$ to match Alg.~\ref{alg.expmF} \\
\midrule
Algorithm~\ref{alg.expmF} & 78 & 89 & -- \\
\expm & 66 & 80 & $\infty$ \\
KL Algorithm & 24 & 51 & 46 \\
\bottomrule
\end{tabular}
\end{table}

The experiments confirm that Algorithm~\ref{alg.expmF} delivers superior accuracy and efficiency compared to both the Kenney and Laub algorithm and \expm, making it a strong candidate for computing the matrix exponential of block triangular matrices.

The MATLAB codes for Algorithm~\ref{alg.expmF} and the Kenney and Laub algorithm are available at \url{https://github.com/aalmohy/expm_block_tri}.

%%%%%%%%%%%%%%%%%%%%%%%%%%
\section{Concluding remarks}
We propose a new algorithm for computing the exponential of block triangular matrices that simultaneously computes the blocks of the matrix exponential without directly involving the full matrix.
We have introduced a linear operator $\cL_f(A, B, E)$, generalizing the Fr\'echet derivative of matrix functions. We explored its algebraic properties and demonstrated how they facilitate the construction of a computational framework for evaluating this operator, specifically for the matrix exponential. Our rigorous backward error analysis not only ensures high accuracy but also inherently promotes computational efficiency through the optimal selection of scaling parameters using sharp bounds. Our numerical experiments demonstrate the high accuracy of the proposed algorithm, while the theoretical analysis highlights its computational advantages over existing methods.

Moreover, we highlighted several practical applications of $\cL_{\exp}$, including its role in exponential integrators for solving systems of stiff or highly oscillatory ordinary differential equations and in computing matrix exponentials for Hamiltonian matrices in control theory. Additionally, we discussed its potential in option pricing models based on polynomial diffusions, where it facilitates the computation of exponentials of recursively nested block upper triangular matrices. Algorithm~\ref{alg.expmF} is a general-purpose method; however, it can serve as a foundation for developing specialized and more efficient algorithms depending on the application and the structure of the inputs.

An open question we raise is whether the bound in Theorem~\ref{Thm.DE} can be expressed in terms of the quantities $\a_p(A)$ and $\a_p(B)$, defined in \eqref{def-alphap}, instead of matrix norms. Achieving this would significantly reduce the scaling parameters and, consequently, the overall computational cost. It is also worth investigating whether the input matrices $A$ and $B$ can be scaled independently using different scaling parameters, which may lead to further efficiency gains and improved numerical behavior.

Finally, future research could explore extending the use of $\cL_f$ to directly compute matrix functions, particularly by employing the Schur decomposition. This provides a promising direction for simplifying and enhancing various computational challenges in scientific computing and numerical methods.
\label{sect.conc}
%%%%%%%%%%%%%%%%%%%%%%%%%%
\section*{Acknowledgments}  
We thank the editor for their professional handling of the review process and the reviewers for their insightful comments and suggestions, which helped improve the presentation of this paper.
%%%%%%%%%%%%%%%%%%%%%%%%%%%%%%%%%%%%%%%%%%%%%%%%%%%%%%%

%\bibliographystyle{myplain2-doi}
%\bibliographystyle{chicago2}
 \bibliographystyle{siamplain}
\bibliography{references}
%%%%%%%%%%%%%%%%%%%%%%%%%%%%%%%%%%%%%%%%%%%%%

%%%%%%%%%%%%%%%%%%%%%%%%%%%%%%%%%%%%%%%%%%%%%

\end{document}